\documentclass[12pt,reqno]{amsart}
\usepackage{amsfonts}
\usepackage{amssymb,color}
\usepackage{amssymb}

\usepackage[colorlinks=true]{hyperref}
\hypersetup{urlcolor=blue, citecolor=blue}

\usepackage[margin=3cm, a4paper]{geometry}

\newtheorem{theo}{Theorem}[section]
\newtheorem{lemm}[theo]{Lemma}
\newtheorem{defi}[theo]{Definition}

\numberwithin{equation}{section}

\newcommand{\bal}{\begin{align}}
\newcommand{\bbal}{\begin{align*}}
\newcommand{\beq}{\begin{equation}}
\newcommand{\eeq}{\end{equation}}
\newcommand{\bca}{\begin{cases}}
\newcommand{\eca}{\end{cases}}

\newcommand{\pa}{\partial}
\newcommand{\fr}{\frac}

\newcommand{\De}{\Delta}

\newcommand{\ep}{\varepsilon}
\newcommand{\dd}{\mathrm{d}}

\newcommand{\R}{\mathbb{R}}

\newcommand{\Z}{\mathbb{Z}}

\begin{document}

\subjclass[2010]{35Q35}
\keywords{Novikov equation, Non-uniform dependence}

\title[Novikov equation]{Non-uniform dependence for the Novikov equation in Besov spaces}

\author[J. Li]{Jinlu Li}
\address{School of Mathematics and Computer Sciences, Gannan Normal University, Ganzhou 341000, China}
\email{lijinlu@gnnu.edu.cn}

\author[M. Li]{Min Li}
\address{School of information Technology, Jiangxi University of Finance and Economics, Nanchang, 330032, China}
\email{limin@jxufe.edu.cn}

\author[W. Zhu]{Weipeng Zhu}
\address{School of Mathematics and Information Science, Guangzhou University, Guangzhou 510006, China}
\email{mathzwp2010@163.com}

\begin{abstract}
In this paper, we investigate the dependence on initial data of solutions to the Novikov equation. We show that the solution map is not uniformly continuous dependence on the initial data in Besov spaces $B^s_{p,r}(\R),\ s>\max\{1+\frac 1p,\frac32\}$.
\end{abstract}

\maketitle

\section{Introduction and main result}

In this paper, we consider the Cauchy problem for the Novikov equation
\begin{equation}\label{nov}
\left\{
\begin{array}{ll}
(1-\partial_x^2)u_t= 3uu_xu_{xx}+u^2u_{xxx}-4u^2u_x,  &t>0,\\[1ex]
u(x,0)=u_0(x).
\end{array}
\right.
\end{equation}
What we are most concerned about is the issue of non-uniform dependence on the initial data. This equation  was discovered very recently by Novikov in a symmetry classification of nonlocal PDEs with
cubic nonlinearity. He showed that Eq.(\ref{nov}) is  integrable  by  using  a definition of the existence of an infinite hierarchy of quasi-local higher symmetries \cite{VN}. It has a bi-Hamiltonian structure and admits  exact peakon solutions $u(t,x)=\pm\sqrt c e^{|x-ct|}$ with $c>0$ \cite{HW}. The Novikov equation had been studied by many authors. Indeed, it is locally well-posed in certain Sobolev spaces and Besov spaces \cite{WY2,WY3,YLZ1,YLZ2}. Moreover, it has global strong solutions \cite{WY2}, finite-time blow up solutions \cite{YLZ2} and global weak solutions \cite{Lai,WY1}.

 The Novikov equation can be thought as a generalization of the well-known Camassa-Holm (CH) equation
  $$ (1-\partial_x^2)u_t=3uu_x-2u_xu_{xx}-uu_{xxx}.$$
  This equation is known as the shallow water wave equation \cite{CH,CL}. It is completely integrable, which has been studied extensively by many authors \cite{CH,Co3,CMcK}.  The CH equation also has a Hamiltonian structure \cite{Co1,FF}, and admits exact peaked solitons of the form $ce^{|x-ct|}$ with
  $c>0$  which are orbitally stable \cite{CS}. These peaked solutions also  mimic the pattern  specific to the waves of greatest height \cite{Co4,CE6,To}.

The local well-posedness for the Cauchy problem of CH equation in Sobolev spaces and Besov spaces was established in \cite{CE2,CE3,Dan2,RB}. Moreover, the CH equation has  global strong solutions \cite{Co2,CE2,CE3}, finite-time blow-up  strong solutions \cite{Co2,CE1,CE2,CE3}, unique global weak solution \cite{XZ}, and it is continuous dependence on initial data \cite{LJY}.

The CH equation was the only known integrable equation
having peakon solutions until 2002 when another such equation was discovered by Degasperis
and Procesi \cite{DP}
  $$ (1-\partial_x^2)u_t=4uu_x-3u_xu_{xx}-uu_{xxx}.$$
The  DP  equation  can  be  regarded  as  a  model  for  nonlinear  shallow  water  dynamics  and  its  asymptotic
accuracy is the same as for the CH shallow water equation\cite{DGH}, also, it's integrable with a bi-Hamiltonian structure \cite{DHH,CIL}.

 Similar to  the CH equation, the  DP  equation has  travelling wave solutions \cite{Lenells,VP}. The Cauchy problem of the DP equation is  locally  well-posed in certain Sobolev spaces and Besov spaces \cite{GL,HH,Y2}. In addition, it has  global strong solutions \cite{LY1,Y2,Y4}, the finite-time blow-up solutions \cite{ELY1,ELY2}, global weak solutions\cite{CK,ELY1,Y3,Y4}, and it is the continuous dependence on initial data \cite{LJY}. Different form the CH equation,  the DP equation has not only peakon solutions \cite{DHH}, periodic peakdon solutions \cite{Y3}, but also  shock peakons\cite{Lun} and the periodic  shock waves \cite{ELY2}.

 The issue of non-uniform dependence has attracted a lot of attention after Kenig et al's study on some dispersive equations \cite{KPV}. For the non-uniform continuity of CH and DP equations in Sobolev spaces, we
refer to \cite{HM,HH1}. And for the Novikov equation, Himonas and Holliman have proved that the data-to-solution map is not uniformly continuous in Sobolev spaces $H^s, s>\frac32,$ they used the method of approximate solutions in conjunction with well-posedness estimates \cite{HH}. Up to now, to our best knowledge, there is no paper concerning the non-uniform dependence on initial data for the Novikov equation under the framework of Besov spaces, which is we shall investigate in this paper.

The Novikov equation (\ref{nov}) can be changed into a transport-like form
\begin{equation}\label{novikov}
\begin{cases}
u_t+u^2u_x=-(1-\pa^2_x)^{-1}\Big(\frac12u_x^3+\pa_x\big(\frac32uu^2_x+u^3\big)\Big),\\
u(0,x)=u_0,
\end{cases}
\end{equation}
For simplicity, we denote
\bbal
R(u)=R_1(u)+R_2(u)+R_3(u),
\end{align*}
where
\bbal
&R_1(u)=-\frac12(1-\pa^2_x)^{-1}\Big(u_x^3\Big), \qquad R_2(u)=-\pa_x(1-\pa^2_x)^{-1}\Big(u^3\Big),
\\&R_3(u)=-\frac32\pa_x(1-\pa^2_x)^{-1}\Big(uu^2_x\Big).
\end{align*}

Then, we have the following result.

\begin{theo}\label{th2}
Let $1\leq p,r\leq \infty$ and $s>\max\{1+\frac 1p,\frac32\}$. The data-to-solution map for the Novikov equation \eqref{novikov} is not uniformly continuous from any bounded subset in $B^s_{p,r}$ into $\mathcal{C}([0,T];B^s_{p,r}(\R))$. That is, there exists two sequences of solutions $u^n$ and $v^n$ such that
\bbal
&||u^n_0||_{B^s_{p,r}(\R)}+||v^n_0||_{B^s_{p,r}(\R)}\lesssim 1, \quad \lim_{n\rightarrow \infty}||u^n_0-v^n_0||_{B^s_{p,r}(\R)}= 0,
\\&\liminf_{n\rightarrow \infty}||u^n(t)-v^n(t)||_{B^s_{p,r}(\R)}\gtrsim t,  \quad t\in[0,T_0],
\end{align*}
with small positive time $T_0$.
\end{theo}

Our paper is organized as follows. In Section 2, we give some preliminaries which will be used in the sequel. In Section 3, we give the proof of our main theorem.\\

\noindent\textbf{Notations.} Given a Banach space $X$, we denote its norm by $\|\cdot\|_{X}$. The symbol $A\lesssim B$ means that there is a uniform positive constant $c$ independent of $A$ and $B$ such that $A\leq cB$.

\section{Littlewood-Paley analysis}

In this section, we will recall some facts about the Littlewood-Paley decomposition, the nonhomogeneous Besov spaces and their some useful properties. For more details, the readers can refer to \cite{B.C.D}.

There exists a couple of smooth functions $(\chi,\varphi)$ valued in $[0,1]$, such that $\chi$ is supported in the ball $\mathcal{B}\triangleq \{\xi\in\mathbb{R}^d:|\xi|\leq \frac 4 3\}$, $\varphi$ is supported in the ring $\mathcal{C}\triangleq \{\xi\in\mathbb{R}^d:\frac 3 4\leq|\xi|\leq \frac 8 3\}$ and $\varphi\equiv 1$ for $\frac{4}{3}\leq |\xi| \leq \frac{3}{2}$. Moreover,
$$\forall\,\ \xi\in\mathbb{R}^d,\,\ \chi(\xi)+{\sum\limits_{j\geq0}\varphi(2^{-j}\xi)}=1,$$
$$\forall\,\ \xi\in\mathbb{R}^d\setminus\{0\},\,\ {\sum\limits_{j\in \Z}\varphi(2^{-j}\xi)}=1,$$
$$|j-j'|\geq 2\Rightarrow\textrm{Supp}\,\ \varphi(2^{-j}\cdot)\cap \textrm{Supp}\,\ \varphi(2^{-j'}\cdot)=\emptyset,$$
$$j\geq 1\Rightarrow\textrm{Supp}\,\ \chi(\cdot)\cap \textrm{Supp}\,\ \varphi(2^{-j}\cdot)=\emptyset.$$
Then, we can define the nonhomogeneous dyadic blocks $\Delta_j$ and nonhomogeneous low frequency cut-off operator $S_j$ as follows:
$$\Delta_j{u}= 0,\,\ if\,\ j\leq -2,\quad
\Delta_{-1}{u}= \chi(D)u=\mathcal{F}^{-1}(\chi \mathcal{F}u),$$
$$\Delta_j{u}= \varphi(2^{-j}D)u=\mathcal{F}^{-1}(\varphi(2^{-j}\cdot)\mathcal{F}u),\,\ if \,\ j\geq 0,$$
$$S_j{u}= {\sum\limits_{j'=-\infty}^{j-1}}\Delta_{j'}{u}.$$

\begin{defi}[\cite{B.C.D}]\label{de2.3}
Let $s\in\mathbb{R}$ and $1\leq p,r\leq\infty$. The nonhomogeneous Besov space $B^s_{p,r}(\R^d)$ consists of all tempered distribution $u$ such that
\begin{align*}
||u||_{B^s_{p,r}(\R^d)}\triangleq \Big|\Big|(2^{js}||\Delta_j{u}||_{L^p(\R^d)})_{j\in \Z}\Big|\Big|_{\ell^r(\Z)}<\infty.
\end{align*}
\end{defi}

Then, we have the following product laws.
\begin{lemm}[\cite{B.C.D}]\label{le-paraduct}
(1) For any $s>0$ and $1\leq p,r\leq\infty$, there exists a positive constant $C=C(d,s,p,r)$ such that
$$\|uv\|_{B^s_{p,r}(\mathbb{R}^d)}\leq C\Big(\|u\|_{L^{\infty}(\mathbb{R}^d)}\|v\|_{B^s_{p,r}(\mathbb{R}^d)}+\|v\|_{L^{\infty}(\mathbb{R}^d)}\|u\|_{B^s_{p,r}(\mathbb{R}^d)}\Big).$$
(2) Let $1\leq p,r\leq\infty$ and $s>\max\{\frac32,1+\frac{d}{p}\}$. Then, we have
$$||uv||_{B^{s-2}_{p,r}(\mathbb{R}^d)}\leq C||u||_{B^{s-1}_{p,r}(\mathbb{R}^d)}||v||_{B^{s-2}_{p,r}(\mathbb{R}^d)}.$$
\end{lemm}

\begin{lemm}[Theorem 3.38, \cite{B.C.D} and Lemma 2.9, \cite{Li-Yin2}]\label{le-trsport}
Let $1\leq p,r\leq \infty$. Assume that
\begin{align}
\sigma> -d \min(\frac{1}{p}, \frac{1}{p'}) \quad \mathrm{or}\quad \sigma> -1-d \min(\frac{1}{p}, \frac{1}{p'})\quad \mathrm{if} \quad \mathrm{div\,} v=0.
\end{align}
There exists a constant $C=C(d,p,r,\sigma)$ such that for any smooth solution to the following linear transport equation:
\begin{equation*}
\quad \partial_t f+v\pa_xf=g,\quad \; f|_{t=0} =f_0.
\end{equation*}
We have
\begin{align}\label{ES2}
\sup_{s\in [0,t]}\|f(s)\|_{B^{\sigma}_{p,r}(\mathbb{R}^d)}\leq Ce^{CV_{p}(v,t)}\Big(\|f_0\|_{B^\sigma_{p,r}(\mathbb{R}^d)}
+\int^t_0\|g(\tau)\|_{B^{s}_{p,r}(\mathbb{R}^d)}\dd \tau\Big),
\end{align}
with
\begin{align*}
V_{p}(v,t)=
\begin{cases}
\int_0^t \|\nabla v(s)\|_{B^{\frac{d}{p}}_{p,\infty}(\mathbb{R}^d)\cap L^\infty(\mathbb{R}^d)}\dd s,&\ \ \mathrm{if} \; \sigma<1+\frac{d}{p},\\
\int_0^t \|\nabla v(s)\|_{B^{\sigma}_{p,r}(\mathbb{R}^d)}\dd s,&\ \ \mathrm{if} \; \sigma=1+\frac{d}{p} \mbox{ and } r>1,\\
\int_0^t \|\nabla v(s)\|_{B^{\sigma-1}_{p,r}(\mathbb{R}^d)}\dd s, &\ \ \mathrm{if} \;\sigma>1+\frac{d}{p}\ \mathrm{or}\ \{\sigma=1+\frac{d}{p} \mbox{ and } r=1\}.
\end{cases}
\end{align*}
If $f=v$, then for all $\sigma>0$ ($\sigma>-1$, if $\mathrm{div\,} v=0$), the estimate \eqref{ES2} holds with
\[V_{p}(t)=\int_0^t \|\nabla v(s)\|_{L^\infty(\mathbb{R}^d)}\dd s.\]
\end{lemm}

\section{Non-uniform continuous dependence}

In this section, we will give the proof of our main theorem. Firstly, we recall the local well-posedness result.
\begin{lemm}[\cite{WY3}]\label{le-wu}
For $1\leq p,r\leq \infty$ and $s>\max\{1+\frac 1p,\frac32\}$ and initial data $u_0\in B^s_{p,r}(\R)$, there exists a time $T=T(s,p,r,||u_0||_{B^s_{p,r}(\R)})>0$ such that the equation \eqref{novikov} have a unique solution $u\in \mathcal{C}([0,T];B^s_{p,r}(\R))$. Moreover, for all $t\in[0,T]$, there holds
\[||u(t)||_{B^s_{p,r}(\R)}\leq C||u_0||_{B^s_{p,r}(\R)}.\]
\end{lemm}

Next, we give two technical lemmas to estimate the error.

Letting $\hat{\phi}$ be a $\mathcal{C}_0(\mathbb{R})$ function such that
\begin{equation*}
\hat{\phi}(x)=
\begin{cases}
1, \quad |x|\leq \frac{1}{4},\\
0, \quad |x|\geq \frac{1}{2}.
\end{cases}
\end{equation*}
We choose the velocity $u^n_0$ having the following form:
\bbal
u^n_0=f_n,
\end{align*}
with
\bbal
f_n(x)=
2^{-ns}\phi(x)\sin(\frac{17}{12}2^nx),
\qquad n \in \Z.
\end{align*}
Notice that
\bbal
\mathrm{supp} \ \hat{f}_n\subset \Big\{\xi\in\R: \ \frac{17}{12}2^n-\fr12\leq |\xi|\leq \frac{17}{12}2^n+\fr12\Big\},
\end{align*}
then we can deduce that $\De_jf_n=0, j\neq n$ and $\De_nf_n=f_n$. Moreover, there holds
\bbal
||f_n||_{B^\sigma_{p,r}}\leq C2^{(\sigma-s)n}.
\end{align*}
Assume that $u^n$ is the solution of \eqref{novikov} with initial data $u^n_0$. Then, we have the following estimate between $u^n_0$ and $u^n$.
\begin{lemm}\label{le-u}
Let
$$\ep_s=\min\Big\{\frac12(s-\max\{1+\frac 1p,\frac32\}),\frac12\Big\}.$$
Then there holds
$$||u^n-u^n_0||_{L^\infty_T(B^s_{p,r})}\leq C2^{-n\ep_s}.$$
\end{lemm}
\begin{proof}
By the well-posedness result (see Lemma \ref{le-wu}), the solution $u^n$ belong to $\mathcal{C}([0,T];B^s_{p,r})$ and have common lifespan $T\simeq 1$. In fact, it is easy to show that for $k\geq -1$,
\bbal
||u^n_0||_{B^{s+k}_{p,r}}\leq C2^{kn},
\end{align*}
and
\bbal
||u^n||_{L^\infty_T(B^{s+k}_{p,r})}\leq C2^{kn}.
\end{align*}
Since $s-\ep_s-1>\frac 1p$, then we have
\bal\label{es-embedding}
||f||_{L^\infty}\leq ||f||_{B^{s-1-\ep_s}_{p,r}}.
\end{align}
Hence, we obtain for all $t\in[0,T]$
\bbal
||u^n-u^n_0||_{B^s_{p,r}}&\leq \int^t_0||\pa_\tau u^n||_{B^s_{p,r}} \dd\tau
\\&\leq \int^t_0||-(1-\pa^2_x)^{-1}\Big(\frac12(u^n_x)^3+\pa_x\big(\frac32u^n(u^n)^2_x+(u^n)^3\big)\Big)||_{B^s_{p,r}} \dd\tau
\\&\quad + \int^t_0||-(u^n)^2 u^n_x||_{B^s_{p,r}} \dd\tau
\\&\leq C\big(||u^n||^2_{B^{s}_{p,r}}||u^n||_{B^{s-1}_{p,r}}+||u^n||^3_{B^{s-1}_{p,r}}+||u^n||_{L^\infty}||u^n||_{B^{s}_{p,r}}||u^n_x||_{B^{s}_{p,r}}\big)
\\&\leq C2^{-n}+C2^{-3n}+C2^{(-1-\ep_s)n}2^n
\\&\leq C2^{-n\ep_s}.
\end{align*}
This completes the proof of this lemma.
\end{proof}

Next, we choose the velocity $v^n_0$ having the following form:
\bbal
v^n_0=f_n+g_n,
\end{align*}
with
\bbal
g_n(x)=
2^{-\frac12n}\phi(x),
\qquad n \in \Z.
\end{align*}
Assume that $v^n$ is the solution of \eqref{novikov} with initial data $v^n_0$. Then, we have the following estimate between $v^n_0$ and $v^n$.
\begin{lemm}\label{le-v}
Under the assumptions of Theorem \ref{th2}, there holds for all $t\in[0,T]$,
$$||v^n-v^n_0+t(v^n_0)^2\pa_xv^n_0||_{B^s_{p,r}}\leq Ct^2+C2^{-n\ep_s}.$$
Here, $C$ is a positive constant independent with $t$.
\end{lemm}
\begin{proof}
By the well-posedness result (see Lemma \ref{le-wu}), the solution $v^n$ belong to $\mathcal{C}([0,T];B^s_{p,r})$ and have common lifespan $T\simeq 1$. For simplicity, we denote $w_n=v^n-v^n_0-tV^n_0$ with $V^n_0=-(v^n_0)^2\pa_x v^n_0$. It is easy to check that for $\sigma\geq -\frac12$
\bal\label{es-(v_0,V_0)}
||v^n_0||_{B^{s+\sigma}_{p,r}}\leq C2^{\sigma n},
\end{align}
then we have
\bal\label{es-(V_0)}
||V^n_0||_{B^{s+\sigma}_{p,r}}&\leq C||v^n_0||_{L^\infty}^2||\pa_x v^n_0||_{B^{s+\sigma}_{p,r}}+C||v^n_0||_{L^\infty}||\pa_x v^n_0||_{L^\infty}||v^n_0||_{B^{s+\sigma}_{p,r}} \nonumber
\\&\leq C2^{-\frac{1}{2}n}2^{-\frac{1}{2}n}2^{(\sigma+1) n}+C2^{-\frac{1}{2}n}2^{-\frac{1}{2}n}2^{\sigma n}\nonumber
\\&\leq C2^{\sigma n}, \quad \mathrm{for} \quad \sigma\geq -\frac12,
\end{align}
and
\bal\label{es-(v)}
||v^n||_{L^\infty_T(B^{s+\sigma}_{p,r})}\leq C2^{\sigma n}, \quad \mathrm{for} \quad \sigma\geq -\frac12.
\end{align}
Note that
$$||v^n_0,V^n_0,v^n||_{B^{s-1}_{p,r}}\leq C||v^n_0,V^n_0,v^n||_{B^{s-\frac{1}{2}}_{p,r}}\leq C 2^{-\frac{1}{2}n},$$
which will be used frequently in the sequel.
By \eqref{novikov}, we can deduce that
\bbal
\pa_tw_n+(v^n)^2\pa_x w_n=&-t(v^n)^2\pa_x V^n_0-t(v^n+v^n_0)V^n_0\pa_xv^n_0
\\&-w_n(v^n+v^n_0)\pa_xv^n_0+R(v^n).
\end{align*}
For the term $R_2(v^n)$, we have from Lemma \ref{le-paraduct} and \eqref{es-(v_0,V_0)}-\eqref{es-(v)} that
\bal\label{es-r2}
||R_2(v^n)||_{B^{s}_{p,r}}&\leq C||(v^n)^3||_{B^{s-1}_{p,r}} \leq C||v^n||_{B^{s-1}_{p,r}}||v^n||^2_{L^\infty}\nonumber
\\&\leq C2^{-\frac32n}\leq C2^{-(1+\ep_s)n}.
\end{align}
For the term $R_3(v^n)$, we obtain from Lemma \ref{le-paraduct} and \eqref{es-embedding}-\eqref{es-(v)},   that
\bal\label{es-r3-(s-1)}
||R_3(v^n)||_{B^{s-1}_{p,r}}&\leq ||v^n(\pa_xv^n)^2||_{B^{s-2}_{p,r}} \nonumber
\\&\leq C||v^n||_{B^{s-1}_{p,r}}||(\pa_xv^n)^2||_{B^{s-2}_{p,r}}\nonumber
\\&\leq C||v^n||_{B^{s-1}_{p,r}}||(\pa_xv^n)^2||_{B^{s-\frac32}_{p,r}} \nonumber
\\&\leq C||v^n||_{B^{s-1}_{p,r}}||\pa_xv^n||_{B^{s-\frac32}_{p,r}}||\pa_xv^n||_{L^\infty} \nonumber
\\&\leq C2^{-(1+\ep_s)n}.
\end{align}
and
\bal\label{es-r3-(s)}
||R_3(v^n)||_{B^{s}_{p,r}}&\leq ||v^n(\pa_xv^n)^2||_{B^{s-1}_{p,r}} \nonumber
\\&\leq C||v^n||_{B^{s}_{p,r}}||v^n||_{L^\infty}||\pa_xv^n||_{L^\infty}+ \nonumber
C||v^n||_{B^{s-1}_{p,r}}||\pa_xv^n||^2_{L^\infty}\nonumber
\\&\leq C2^{-(\frac12+\ep_s)n}\leq C2^{-n\ep_s}.
\end{align}
For the term $R_1(v^n)$, we rewrite it as
\bbal
R_1(v^n)=&\underbrace{-(1-\pa^2_x)^{-1}\Big(\pa_xv^n(\pa_xv^n+\pa_xv^n_0)\pa_x w_n\Big)}_{R_{1,1}}
\\&\underbrace{-(1-\pa^2_x)^{-1}\Big(t(\pa_xv^n)(\pa_xv^n+\pa_xv^n_0)\pa_x V^n_0\Big)}_{R_{1,2}}
\\&\underbrace{-(1-\pa^2_x)^{-1}\Big(\pa_xv^n(\pa_xv^n_0)^2\Big)}_{R_{1,3}},
\end{align*}
then we have from Lemma \ref{le-paraduct} that
\bal\label{es-r11}
||R_{1,1}||_{B^{s}_{p,r}}&\leq \nonumber C||\pa_xv^n(\pa_xv^n+\pa_xv^n_0)\pa_xw_n||_{B^{s-2}_{p,r}}
\\&\leq C||\pa_xw^n||_{B^{s-2}_{p,r}}||\pa_xv^n(\pa_xv^n+\pa_xv^n_0)||_{B^{s-1}_{p,r}} \nonumber
\\&\leq C||w^n||_{B^{s-1}_{p,r}},
\end{align}
\bal\label{es-r11}
||R_{1,2}||_{B^{s}_{p,r}}&\leq Ct||\pa_xv^n(\pa_xv^n+\pa_xv^n_0)\pa_xV^n_0||_{B^{s-2}_{p,r}} \nonumber
\\&\leq Ct||\pa_xV^n_0||_{B^{s-2}_{p,r}}||\pa_xv^n(\pa_xv^n+\pa_xv^n_0)||_{B^{s-1}_{p,r}} \nonumber
\\&\leq Ct2^{-\frac{1}{2}n}.
\end{align}
Using the following inequality
\bbal
||(\pa_xv^n_0)^3||_{B^{s-2}_{p,r}}&\leq C||(\pa_xg_n)^2\pa_xf_n||_{B^{s-2}_{p,r}}+ C||(\pa_xf_n)^2\pa_xg_n||_{B^{s-1}_{p,r}}
\\&\quad +C||(\pa_xf_n)^3||_{B^{s-1}_{p,r}}
+C||(\pa_xg_n)^3||_{B^{s-1}_{p,r}}
\\&\leq C2^{n(s-2)}||\pa_x g_n||^2_{L^\infty}||\pa_x f_n||_{L^p}+C||g_n||_{B^s_{p,r}}||f_n||_{B^s_{p,r}}||\pa_xf_n||_{L^\infty}
\\&\quad +C||g_n||^3_{B^{s}_{p,r}}+
C||f_n||_{B^s_{p,r}}||\pa_xf_n||^2_{L^\infty}
\\&\leq C2^{-2n}+C2^{-\frac32n}+C2^{-2n(s-1)}+C2^{-n(s-1)-\frac12n}
\\&\leq C2^{-n(1+\ep_s)},
\end{align*}
we have
\bal\label{es-r13}
||R_{1,3}||_{B^{s}_{p,r}}&\leq C||\pa_xv^n(\pa_xv^n_0)^2||_{B^{s-2}_{p,r}}\nonumber
\\&\leq  \nonumber C||\pa_xw^n(\pa_xv^n_0)^2||_{B^{s-2}_{p,r}}+Ct||\pa_xV^n_0(\pa_xv^n_0)^2||_{B^{s-2}_{p,r}}
+C||\pa_xv^n_0(\pa_xv^n_0)^2||_{B^{s-2}_{p,r}}\nonumber
\\&\leq Ct2^{-\frac{1}{2}n}+C||w_n||_{B^{s-1}_{p,r}}+2^{-n(1+\ep_s)}.
\end{align}
Combining \eqref{es-r2}-\eqref{es-r13}, we obtain
\bal\label{Es-r-(s-1)}
||R(v^n)||_{B^{s-1}_{p,r}}\leq Ct2^{-\frac{1}{2}n}+C||w_n||_{B^{s-1}_{p,r}}+2^{-n(1+\ep_s)},
\end{align}
and
\bal\label{Es-r-(s)}
||R(v^n)||_{B^{s}_{p,r}}\leq Ct2^{-\frac{1}{2}n}+C||w_n||_{B^{s-1}_{p,r}}+2^{-n\ep_s}.
\end{align}
By Lemma \ref{le-paraduct}, we have
\bal\label{Es-V1-1}
||(v^n)^2\pa_x V^n_0||_{B^{s-1}_{p,r}}&\leq  C||v^n||^2_{B^{s-1}_{p,r}}||V^n_0||_{B^{s}_{p,r}}
\leq C2^{-n},
\end{align}

\bal\label{Es-V1}
||(v^n)^2\pa_x V^n_0||_{B^{s}_{p,r}}&\leq \nonumber C||v^n||^2_{L^\infty}||V^n_0||_{B^{s+1}_{p,r}}+C||v^n||^2_{B^{s}_{p,r}}||\pa_xV^n_0||_{L^\infty}
\\&\leq C,
\end{align}
\bal\label{Es-V2-2}
||(v^n+v^n_0)V^n_0\pa_x v^n_0||_{B^{s-1}_{p,r}}\leq C2^{-n},
\end{align}
\bal\label{Es-V2}
||(v^n+v^n_0)V^n_0\pa_x v^n_0||_{B^{s}_{p,r}}\leq C.
\end{align}
Applying Lemma \ref{le-paraduct} again yields
\bal\label{Es-w-(s-1)}
||w_n(v^n+v^n_0)\pa_xv^n_0||_{B^{s-1}_{p,r}}\leq C||w_n||_{B^{s-1}_{p,r}},
\end{align}
and
\bal\label{Es-w-(s)}
||w_n(v^n+v^n_0)\pa_xv^n_0||_{B^{s}_{p,r}}&\leq C||w_n||_{B^{s-1}_{p,r}}||v^n_0||_{B^{s+1}_{p,r}}||v^n+v^n_0||_{B^{s-1}_{p,r}}+C||w_n||_{B^{s}_{p,r}}\nonumber
\\&\leq C2^{\frac{1}{2}n}||w_n||_{B^{s-1}_{p,r}}+C||w_n||_{B^{s}_{p,r}}.
\end{align}
According to Lemma \ref{le-trsport} and combining \eqref{Es-r-(s-1)}, \eqref{Es-V1-1}, \eqref{Es-V2-2}, \eqref{Es-w-(s-1)}, we obtain
\bbal
||w_n||_{B^{s-1}_{p,r}}&\leq C\int^t_0||w_n||_{B^{s-1}_{p,r}}\dd \tau+Ct^22^{-\frac{1}{2}n}+C2^{-n(1+\ep_s)},
\end{align*}
which implies
\bal\label{Es-con-(s-1)}
||w_n||_{B^{s-1}_{p,r}}\leq Ct^2 2^{-\frac{1}{2}n}+C2^{-n(1+\ep_s)}.
\end{align}
Using Lemma \ref{le-trsport} again and combining \eqref{Es-r-(s)}, \eqref{Es-V1}, \eqref{Es-V2}, \eqref{Es-w-(s)}, \eqref{Es-con-(s-1)} we have
\bbal
||w_n||_{B^{s}_{p,r}}&\leq C\int^t_0||w_n||_{B^{s}_{p,r}}\dd \tau+\int^t_02^{\frac{1}{2}n}||w_n||_{B^{s-1}_{p,r}}\dd \tau+Ct^2+C2^{-n\ep_s}
\\&\leq C\int^t_0||w_n||_{B^{s}_{p,r}}\dd \tau+Ct^2+C2^{-n\ep_s},
\end{align*}
which implies
\bal\label{Es-con-(s)}
||w_n||_{B^{s}_{p,r}}\leq Ct^2+C2^{-n\ep_s}.
\end{align}
This completes the proof of this lemma.
\end{proof}

\textbf{Proof of the main theorem.}
Now, we need prove the result of Theorem \ref{th2}. It is easy to show that
\bal\label{con-small}
||u^n_0-v^n_0||_{B^s_{p,r}}\leq ||g_n||_{B^s_{p,r}}\leq C2^{-\frac12n},
\end{align}
which tend to 0 for $n$ tends to infinity. According Lemmas \ref{le-u}-\ref{le-v} , we have
\bal\label{Es-cha}
||u^n-v^n||_{B^s_{p,r}}\geq c||t(v^n_0)^2\pa_x v^n_0||_{B^s_{p,r}}-Ct^2-C2^{-n\ep_s}.
\end{align}
Notice that
\bbal
(v^n_0)^2\pa_x v^n_0&=(v^n_0)^2\pa_x g_n+(f_n)^2\pa_x f_n+2g_nf_n\pa_x f_n+(g_n)^2\pa_x f_n.
\end{align*}
By Lemma \ref{le-paraduct}, we have
\bbal
||(v^n_0)^2\pa_x g_n||_{B^s_{p,r}}&\leq C||v^n_0||_{L^\infty}||v^n_0||_{B^{s}_{p,r}}||g_n||_{B^{s+1}_{p,r}}\leq C2^{-n\ep_s},
\end{align*}
\bbal
||(f_n)^2\pa_x f_n||_{B^s_{p,r}}&\leq C||f_n||^2_{L^\infty}||f_n||_{B^{s+1}_{p,r}}+C||f_n||_{L^\infty}||\pa_xf_n||_{L^\infty}||f_n||_{B^{s}_{p,r}}
\\&\leq C2^{-n\ep_s},
\end{align*}
and
\bbal
||g_nf_n\pa_x f_n||_{B^s_{p,r}}&\leq C||f_n||_{L^\infty}||g_n||_{L^\infty}||f_n||_{B^{s+1}_{p,r}}+C||f_n||_{B^s_{p,r}}||\pa_xf_n||_{L^\infty}||g_n||_{B^{s}_{p,r}}
\\&\leq C2^{-n\ep_s}.
\end{align*}
This alongs with \eqref{Es-cha} implies
\bal\label{con-large}
||u^n-v^n||_{B^s_{p,r}}\geq ct||(g_n)^2\pa_x f_n||_{B^s_{p,r}}-Ct^2-C2^{-n\ep_s}.
\end{align}
Using the facts $\De_j((g_n)^2\pa_x f_n)=0,j\neq n$ and $\De_n((g_n)^2\pa_x f_n)=(g_n)^2\pa_x f_n$, direct calculation shows that
\bal\label{con-large2}
||(g_n)^2\pa_x f_n&||_{B^s_{p,r}}=2^{ns}||(g_n)^2\pa_x f_n||_{L^p}\nonumber
\\&=\frac{17}{12}||\phi^3(x)\cos(\frac{17}{12}2^nx)||_{L^p}\rightarrow \frac{17}{12} \Big(\frac{\int^\pi_0|\cos x|^p\dd x}{\pi}\Big)^{\frac1p}||\phi^3(x)||_{L^p},
\end{align}
by Riemann-Lebesgue's Lemma. Hence, combining \eqref{con-small}, \eqref{con-large}, \eqref{con-large2} and choosing the positive $T_0$ small enough, we can obtain our result.

\vspace*{1em}
\noindent\textbf{Acknowledgements.}  J. Li is supported by the National Natural Science Foundation of China (Grant No.11801090).  W. Zhu is partially supported by the National Natural Science Foundation of China (Grant No.11901092) and Natural Science Foundation of Guangdong Province (No.2017A030310634).


\begin{thebibliography}{99}
\linespread{0}

\bibitem{B.C.D} H. Bahouri, J. Y. Chemin and R. Danchin, {Fourier Analysis and Nonlinear Partial Differential Equations}, {Grundlehren der Mathematischen Wissenschaften}, vol. 343, Springer-Verlag, Berlin, Heidelberg, 2011.




\bibitem{CH}
R. Camassa and D. D. Holm, An integrable shallow water equation with peaked solitons, {\it  Phys. Rev.
Lett.}, {\bf 71} (1993), 1661--1664.

\bibitem{CK}
G. M. Coclite and K. H. Karlsen, On the well-posedness of the
Degasperis-Procesi equation,  {\it J. Funct. Anal.}, {\bf 233} (2006), 60--91.
\bibitem{Co1}
A. Constantin, The Hamiltonian structure of the Camassa-Holm equation, {\it Expositiones Mathematicae}, 15(1) (1997), 53--85.
\bibitem{Co2}
A. Constantin,  Existence of permanent and breaking waves for a
shallow water equation: a geometric approach, {\it Ann. Inst.
Fourier (Grenoble)}, {\bf 50} (2000), 321--362.
\bibitem{Co3}
A. Constantin,  On the  scattering problem for the Camassa-Holm equation,  {\it Proceedings of the  Royal Society of  London, Series A}, 457(2001), 953--970.
\bibitem{Co4}
A. Constantin, The trajectories of particles in Stokes waves.
{\it Invent. Math.}, 166 (2006), no. 3, 523--535.
\bibitem{CE1}
A. Constantin and J. Escher, Wave breaking for nonlinear nonlocal
shallow water equations, {\it Acta Math.,} {\bf 181} (1998),
229--243.
\bibitem{CE2}
A. Constantin and J. Escher, Global existence and blow-up for a
shallow water equation, {\it Ann. Scuola Norm. Sup. Pisa Cl. Sci. (4),} {\bf 26}
(1998), 303--328.
\bibitem{CE3}
A. Constantin and J. Escher, Well-posedness, global existence and
blowup phenomena for a periodic quasi-linear hyperbolic equation,
{\it Comm. Pure Appl. Math.,} {\bf 51} (1998), 475--504.
\bibitem{CE6}
A. Constantin and J.  Escher,
Analyticity of periodic traveling free surface water waves with vorticity.
{\it Ann. of Math.}, (2) 173 (2011), no. 1, 559¨C568.
\bibitem{CIL}
A.Constantin, R. Ivanov,  J. Lenells,
Inverse scattering transform for the Degasperis-Procesi equation.
{\it Nonlinearity } 23 (2010), no. 10, 2559--2575.
\bibitem{CL}
A. Constantin and D. Lannes,
The hydrodynamical relevance of the Camassa-Holm and Degasperis-Procesi equations.
{\it Arch. Ration. Mech. Anal.}, 192 (2009), no. 1, 165--186.
\bibitem{CMcK}
A. Constantin and  H. P. McKean,
A shallow water equation on the circle.
{\it Comm. Pure Appl. Math.}, 52 (1999), no. 8, 949--982.
\bibitem{CS}
 A. Constantin and W. A. Strauss, Stability of peakons, {\it Comm. Pure Appl. Math. }, 53 (2000), 603--610.
\bibitem{Dan2}
R. Danchin, A few remarks on the Camassa-Holm equation, {\it
Differential and Integral Equations}, {\bf 14} (2001), 953--988.
\bibitem{DHH}
A. Degasperis, D. D. Holm, and A. N. W. Hone, A new integral equation with peakon solutions,
{\it Theor. Math. Phys.}, 133 (2002), 1463--1474.
\bibitem{DP}
A. Degasperis and M. Procesi, Asymptotic integrability, {\it Symmetry and perturbation theory (Rome 1998)},
 pp. 23--37. World Sci. Publ., River Edge, NJ, 1999.
\bibitem{DGH}
 H. R. Dullin, G. A. Gottwald, and D. D. Holm, On asymptotically equivalent shallow water wave
equations, {\it Phys. D}, 190 (2004), 1--14.
\bibitem{ELY1}
J. Escher, Y. Liu and Z. Yin, Global weak solutions and blow-up structure foe the Degasperis-Procesi  equation, {\it J. Funct. Anal.}, {\bf 241} (2006), 457--485.
\bibitem{ELY2}
J. Escher, Y. Liu and Z. Yin, Shock waves and blow-up phenomena for the periodic Degasperis-Procesi  equation, {\it Indiana Univ. Math. J.}, {\bf 56} (2007), 87--177.
\bibitem{FF}
  A. Fokas and B. Fuchssteiner, Symplectic structures, their Backlund transformation and hereditary
symmetries, {\it Physica D}, 4(1) (1981/82), 47--66.
\bibitem{GL}
G. Gui and Y. Liu, On the Cauchy problem for the Degasperis-Procesi equation,{\it Quart. Appl.
Math.}, 69 (2011), 445--464.
\bibitem{HH}
A. A. Himonas and C. Holliman, The Cauchy problem for the Novikov equation,{\it  Nonlinearity}, 25
(2012), 449--479.
\bibitem{HH1}
A. Himonas and C. Holliman, On well-posedness of the Degasperis-Procesi equation {Discrete Contin. Dyn. Syst. A}, 31 (2011) 469--484.
\bibitem{HM}
A. Himonas and G. Misio lek, Non-uniform dependence on initial data of solutions to the Euler
equations of hydrodynamics, {Comm. Math. Phys.}, {296} (2010), 285--301.
\bibitem{HW}
 A. N. W. Hone and J. Wang, Integrable peakon equations with cubic nonlinearity, {\it Journal of
Physics A: Mathematical and Theoretical}, 41 2008, 372002, 10pp.

\bibitem{KPV}
C. Kenig, G. Ponce, L. Vega, On the ill-posedness of some canonical dispersive equations,
{Duke Math.}, {106} (2001) 617--633.
\bibitem{Lai}
 S. Lai, Global weak solutions to the Novikov equation, {\it J. Funct. Anal.}, 265 (2013),
520--544.
\bibitem{Lenells}
J. Lenells, Traveling wave solutions of the Degasperis-Procesi equation, {\it J. Math. Anal. Appl.}, 306
(2005), 72--82.
\bibitem{LJY}
J. Li and Z. Yin, Remarks on the well-posedness of Camassa-Holm type equations in Besov
spaces, {J. Differential Equations}, 261 (2016), 6125--6143.
\bibitem{Li-Yin2} J. Li  and Z. Yin, \textit{Well-posedness and analytic solutions of the two-component Euler-Poincar\'{e} system}, {Monatsh. Math.}, 183 (2017), 509--537.
\bibitem{LY1}
Y. Liu and Z. Yin, Global Existence and Blow-Up Phenomena for the
Degasperis-Procesi Equation, {\it Commun. Math. Phys.}, {\bf 267}
(2006), 801--820.
\bibitem{Lun}
 H. Lundmark, Formation and dynamics of shock waves in the Degasperis-Procesi equation,
 {\it J. Nonlinear. Sci.}, 17 (2007), 169--198.
\bibitem{VN}
V. Novikov, Generalization of the Camassa-Holm equation, {\it J. Phys. A}, 42 (2009), 342002, 14 pp.
\bibitem{RB}
G. Rodriguez-Blanco, On the Cauchy problem for the Camassa-Holm
equation, {\it Nonlinear Anal.,} {\bf 46} (2001), 309--327.

\bibitem{To}
J. F. Toland,  Stokes waves. {\it Topol. Methods Nonlinear Anal.}, 7 (1996), no.1, 1--48.
\bibitem{VP}
 V. O. Vakhnenko and E. J. Parkes, Periodic and solitary-wave solutions of the Degasperis-Procesi
equation, {\it Chaos Solitons Fractals}, 20 (2004), 1059--1073.
\bibitem{WY1}
X. Wu and Z. Yin, Global weak solutions for the Novikov equation,{\it  Journal of Physics A: Mathe-
matical and Theoretical}, 44 (2011), 055202, 17pp.
\bibitem{WY2}
X. Wu and Z. Yin, Well-posedness and global existence for the Novikov equation, {\it Annali della
Scuola Normale Superiore di Pisa. Classe di Scienze. Serie V}, 11 (2012), 707--727.
\bibitem{WY3}
 X. Wu and Z. Yin, A note on the Cauchy problem of the Novikov equation, {\it Applicable Analysis},
92 (2013), 1116--1137.
\bibitem{XZ}
Z. Xin and P. Zhang, On the weak solutions to a shallow water
equation, {\it Comm. Pure Appl. Math.,} {\bf 53} (2000),
1411--1433.
\bibitem{YLZ1}
 W. Yan, Y. Li and Y. Zhang, The Cauchy problem for the integrable Novikov equation, {\it J. Differential Equations}, 253 (2012), 298--318.
\bibitem{YLZ2}
 W. Yan, Y. Li and Y. Zhang, The Cauchy problem for the Novikov equation, {\it Nonlinear Differential
Equations and Applications NoDEA}, 20 (2013), 1157--1169.
\bibitem{Y2}
Z. Yin, Global existence for a new periodic integrable equation,
{\it J. Math. Anal. Appl.}, {\bf 49} (2003), 129--139.
\bibitem{Y3}
Z. Yin, Global weak solutions to a new periodic integrable equation
with peakon solutions, {\it J. Funct. Anal.}, {\bf 212} (2004),
182--194.
\bibitem{Y4}
Z. Yin, Global solutions to a new integrable equation with peakons,
{\it Indiana Univ. Math. J.}, {\bf 53} (2004), 1189--1210.


\end{thebibliography}
\end{document}